\documentclass[11pt]{amsart}

\usepackage[T1]{fontenc}
\usepackage[utf8]{inputenc}
\usepackage{lmodern}
\usepackage[english]{babel}
\usepackage[autostyle]{csquotes}

\usepackage[pdfusetitle,linktocpage,colorlinks=false]{hyperref}
\usepackage{doi}
\usepackage{filecontents}
\usepackage{enumerate,enumitem}

\usepackage{amsmath,amssymb,amsthm}

\newtheorem{thm}{Theorem}[section]
\newtheorem{prop}[thm]{Proposition}
\newtheorem{lem}[thm]{Lemma}
\newtheorem{rem}[thm]{Remark}
\newtheorem{cor}[thm]{Corollary}

\newtheorem{defn}[thm]{Definition}
\newcommand{\C}{\mathbb{C}}
\newcommand{\N}{\mathbb{N}}
\newcommand{\Z}{\mathbb{Z}}

\AtEndDocument{\bigskip{\footnotesize%
  \textsc{Institute of Mathematics of the Polish Academy of Sciences, \'{S}niadeckich 8, 00-656 Warsaw, Poland} \par  
  \textit{E-mail address}: \texttt{ychung@impan.pl} \par
  \addvspace{\medskipamount}
  \textsc{Mathematisches Institut der WWU M\"{u}nster, Einsteinstrasse 62, 48149 M\"{u}nster, Germany} \par
  \textit{E-mail address:} \texttt{lik@uni-muenster.de} \par
}}

\begin{document}

\title{Rigidity of $\ell^p$ Roe-type algebras}

\author{Yeong Chyuan Chung$^{1}$ \and Kang Li$^{2}$}

\thanks{{$^{1}$} Supported by the European Research Council (ERC-677120).}

\thanks{{$^{2}$} Supported by the Danish Council for Independent Research (DFF-5051-00037) and partially supported by the DFG (SFB 878).}

\begin{abstract}
We investigate the rigidity of the $\ell^p$ analog of Roe-type algebras. In particular, we show that if $p\in[1,\infty)\setminus\{2\}$, then an isometric isomorphism between the $\ell^p$ uniform Roe algebras of two metric spaces with bounded geometry yields a bijective coarse equivalence between the underlying metric spaces, while a stable isometric isomorphism yields a coarse equivalence. We also obtain similar results for other $\ell^p$ Roe-type algebras. In this paper, we do not assume that the metric spaces have Yu's property A or finite decomposition complexity.
\end{abstract}

\date{\today}
\maketitle

\parskip 4pt



\section{Introduction}

Roe-type algebras are $C^*$-algebras associated to discrete metric spaces and they encode the coarse (or large-scale) geometry of the underlying metric spaces, and they have been well-studied, providing a link between coarse geometry of metric spaces and operator algebra theory (e.g., \cite{ALLW17, MR1876896, MR1739727, MR3158244, LL, LW18, MR1763912, MR2873171, Scarparo:2016kl, MR1905840, MR2800923, WZ10}). They feature in the (uniform) coarse Baum-Connes conjecture (e.g., \cite{MR1905840, MR2523336, Yu95, MR1451759, Yu00}), and have also recently been found to be useful in the study of topological phases of matter (e.g., \cite{Kub,EwertMeyer}). A natural question that arises is about the rigidity of these algebras, namely how well the $C^\ast$-algebra encodes the coarse geometry of the metric space, or whether the coarse geometry of a metric space can be recovered from its Roe algebra, and this was well studied in \cite{MR3116573, WW}. 

We will consider metric spaces with bounded geometry defined as follows: 

\begin{defn}
Let $X$ be a metric space. Then $X$ is said to have \emph{bounded geometry} if for all $R\geq 0$ there exists $N_R\in\mathbb{N}$ such that for all $x\in X$, the ball of radius $R$ about $x$ has at most $N_R$ elements.
\end{defn}

Note that every metric space with bounded geometry is necessarily countable and discrete. The simplest examples of such metric spaces are finitely generated discrete groups equipped with word metrics. Other interesting examples are box spaces of finitely generated residually finite groups (see e.g. \cite[Definition~6.3.2]{MR2562146}). We are particularly interested in (bijective) coarse equivalence classes of such metric spaces in the following sense:

\begin{defn}
Let $X$ and $Y$ be metric spaces. 
\begin{itemize}
\item A (not necessarily continuous) map $f:X\rightarrow Y$ is said to be \emph{uniformly expansive} if for all $R>0$ there exists $S>0$ such that if $x_1,x_2\in X$ satisfy $d(x_1,x_2)\leq R$, then $d(f(x_1),f(x_2))\leq S$.

\item Two maps $f,g:X\rightarrow Y$ are said to be \emph{close} if there exists $C>0$ such that $d(f(x),g(x))\leq C$ for all $x\in X$.

\item Two metric spaces $X$ and $Y$ are said to be \emph{coarsely equivalent} if there exist uniformly expansive maps $f:X\rightarrow Y$ and $g:Y\rightarrow X$ such that $f\circ g$ and $g\circ f$ are close to the identity maps, respectively. In this case, we say both $f$ and $g$ are \emph{coarse equivalences} between $X$ and $Y$.

\item We say a map $f:X\to Y$ is a \emph{bijective coarse equivalence} if $f$ is both a coarse equivalence and a bijection. In this case, we say $X$ and $Y$ are \emph{bijectively coarsely equivalent}.

\end{itemize}
\end{defn}

It was shown in \cite[Theorem 4]{BNW} that if $X$ and $Y$ are coarsely equivalent metric spaces with bounded geometry, then their uniform Roe $C^\ast$-algebras are Morita equivalent\footnote{Note that for $\sigma$-unital (in particular unital) $C^\ast$-algebras, being Morita equivalent is the same as being stably $*$-isomorphic \cite[Theorem 1.2]{BGR}.}. A partial converse was obtained in \cite{MR3116573} under the assumption that the metric spaces have Yu's property A (see \cite[Definition 2.1]{Yu00}), which can be regarded as a coarse variant of amenability. In fact, it was shown in \cite{MR3116573} that under the assumption of property A, $*$-isomorphisms between Roe-type $C^\ast$-algebras yield coarse equivalences:

\begin{thm}\cite[Theorem 4.1, Theorem 6.1 and Corollary 6.2]{MR3116573}
Let $X$ and $Y$ be metric spaces with bounded geometry and property A. 
\begin{enumerate}
\item $X$ and $Y$ are coarsely equivalent if and only if their uniform Roe $C^\ast$-algebras are stably $*$-isomorphic.
\item If $A^\ast(X)$ and $A^\ast(Y)$ are Roe-type $C^\ast$-algebras associated to $X$ and $Y$ respectively, and there is a $*$-isomorphism between $A^\ast(X)$ and $A^\ast(Y)$, then $X$ and $Y$ are coarsely equivalent. The same conclusion holds for the stable uniform Roe $C^\ast$-algebra.

The converse is true for $UC^\ast(X)$, $C^\ast(X)$, and $C^\ast_s(X)$.
\end{enumerate}
\end{thm}

If there is a bijective coarse equivalence between two metric spaces with bounded geometry, then the proof of \cite[Theorem 4]{BNW} or \cite[Proposition~2.3]{LL} shows that their uniform Roe $C^\ast$-algebras are $*$-isomorphic. A partial converse to this statement was obtained in \cite{WW} under the assumption that the metric spaces have finite decomposition complexity as defined in \cite{GTY}.

\begin{thm}\cite[Corollary 1.16]{WW}
Let $X$ and $Y$ be metric spaces with bounded geometry. Suppose $X$ and $Y$ have finite decomposition complexity. Then $X$ and $Y$ are bijectively coarsely equivalent if and only if their uniform Roe $C^\ast$-algebras are $*$-isomorphic.
\end{thm}

In the purely algebraic setting, it was shown in \cite[Theorem~5.1 and Theorem~6.1]{MR3116573}, \cite{WW} and \cite[Theorem~8.1]{BF18} that the results above hold without requiring property A or finite decomposition complexity as follows:
\begin{thm}
Let $X$ and $Y$ be metric spaces with bounded geometry. 
\begin{enumerate}
\item $X$ and $Y$ are bijectively coarsely equivalent if and only if their algebraic uniform Roe algebras are $*$-isomorphic.
\item If $A[X]$ and $A[Y]$ are algebraic Roe-type algebras associated to $X$ and $Y$ respectively, and there is a $*$-isomorphism between $A[X]$ and $A[Y]$, then $X$ and $Y$ are coarsely equivalent. The same conclusion holds for the algebraic stable uniform Roe algebra.

The converse is true for $U\mathbb{C}[X]$, $\mathbb{C}[X]$, and $\mathbb{C}_s[X]$.
\end{enumerate}
\end{thm}

In this paper, we will study the rigidity problem for the $\ell^p$ analog of uniform Roe algebras and other Roe-type algebras for $1\leq p<\infty$.

\begin{defn} \label{pRoedef}
Let $(X,d)$ be a metric space with bounded geometry and $1\leq p<\infty$. For an operator $T=(T_{xy})_{x,y\in X}\in B(\ell^p(X))$, where $T_{xy}=(T\delta_y)(x)$, we define the propagation of $T$ to be
\[ \mathop{\rm prop}(T)=\sup\{ d(x,y):x,y\in X,T_{xy}\neq 0 \}\in[0,\infty]. \]
We denote by $\mathbb{C}_u^p[X]$ the unital algebra of all bounded operators on $\ell^p(X)$ with finite propagation. The $\ell^p$ uniform Roe algebra, denoted by $B^p_u(X)$, is defined to be the operator norm closure of $\mathbb{C}_u^p[X]$ in $B(\ell^p(X))$.
\end{defn}

The $\ell^p$ uniform Roe algebra belongs to a class of algebras that we may call $\ell^p$ Roe-type algebras (see Definition \ref{Roetype}). Such algebras were considered in \cite{MR3557774} in connection with criteria for Fredholmness. Other examples of $\ell^p$ Roe-type algebras that may be of particular interest are the following:
\begin{enumerate} \label{egRoetype}
\item The $\ell^p$ uniform algebra of $X$, denoted $UB^p(X)$, is the operator norm closure of the algebra $U\mathbb{C}^p[X]$ of all finite propagation bounded operators $T$ on $\ell^p(X,\ell^p)$ such that there exists $N\in\mathbb{N}$ such that for all $x,y\in X$, we have that $T_{xy}$ is an operator on $\ell^p$ of rank at most $N$.
\item The $\ell^p$ Roe algebra of $X$, denoted $B^p(X)$, is the operator norm closure of the algebra $\mathbb{C}^p[X]$ of all finite propagation bounded operators $T$ on $\ell^p(X,\ell^p)$ such that for all $x,y\in X$, we have $T_{xy}\in\overline{M}^p_\infty$, where $\overline{M}^p_\infty=\overline{\bigcup_{n\in\mathbb{N}}M_n(\mathbb{C})}\subset B(\ell^p(\mathbb{N}))$.
\end{enumerate}

An alternative definition of the $\ell^p$ Roe algebra is to require $T_{xy}\in K(\ell^p(\N))$ for all $x,y\in X$. When $p\in(1,\infty)$, we have $\overline{M}^p_\infty=K(\ell^p(\N))$ so we get the same algebra as above. However, when $p=1$, $\overline{M}^1_\infty$ is strictly contained in $K(\ell^1(\N))$. Nevertheless, we still get an $\ell^p$ Roe-type algebra that is a coarse invariant, and our rigidity result still holds for this algebra.

Another related example that is not quite a Roe-type algebra but exhibits similar behavior is the stable $\ell^p$ uniform Roe algebra of $X$, denoted $B^p_s(X)$, and is the operator norm closure of the algebra $\mathbb{C}^p_s[X]$ of all finite propagation bounded operators $T$ on $\ell^p(X,\ell^p)$ such that there exists a finite-dimensional subspace $E_T\subset\ell^p$ such that for all $x,y\in X$, we have $T_{xy}\in B(E_T)$. 
One can easily check that it is isomorphic to $B^p_u(X)\otimes K(\ell^p(\N))$ for $p\in[1,\infty)$.

We can also give an alternative definition of the stable $\ell^p$ uniform Roe algebra by requiring the existence of some $k\in \N$ such that $T_{xy}\in M_k(\C)$ for all $x,y\in X$. This algebra will then be isomorphic to $B^p_u(X)\otimes\overline{M}^p_\infty$ for $p\in[1,\infty)$, and is different from the algebra defined above only when $p=1$.

It may be worth noting that $\ell^1$ Roe-type algebras are structurally different from $\ell^p$ Roe-type algebras for $p\in(1,\infty)$ in that $\ell^1$ Roe-type algebras may not contain some finite rank operators while $\ell^p$ Roe-type algebras contain all finite rank operators when $p\in(1,\infty)$. An example is given in Remark \ref{rankone} for the $\ell^1$ uniform Roe algebra.

The following is a summary of the main results of this paper:

\begin{thm} (see Theorem \ref{thm1}, Theorem \ref{thm2}, and Theorem \ref{thm:Roetype})
Let $X$ and $Y$ be metric spaces with bounded geometry, and let $p\in [1,\infty)\setminus\{2\}$.
\begin{enumerate}
\item $X$ and $Y$ are bijectively coarsely equivalent if and only if their $\ell^p$ uniform Roe algebras are isometrically isomorphic.
\item $X$ and $Y$ are coarsely equivalent if and only if their $\ell^p$ uniform Roe algebras are stably isometrically isomorphic.
\item If $A^p(X)$ and $A^p(Y)$ are $\ell^p$ Roe-type algebras associated to $X$ and $Y$ respectively, and there is an isometric isomorphism between $A^p(X)$ and $A^p(Y)$, then $X$ and $Y$ are coarsely equivalent. The same conclusion holds if $B^p_s(X)$ and $B^p_s(Y)$ are isometrically isomorphic.

The converse is true for $UB^p(X)$, $B^p(X)$, and $B^p_s(X)$.
\end{enumerate}
\end{thm}

\begin{cor} (see Corollary \ref{cor2} and Corollary \ref{cor:Roetype})
Let $X$ and $Y$ be metric spaces with bounded geometry, and let $p\in [1,\infty)$.
\begin{enumerate}
\item $X$ and $Y$ are bijectively coarsely equivalent if and only if $\mathbb{C}^p_u[X]$ and $\mathbb{C}^p_u[Y]$ are isometrically isomorphic.
\item If $A^p[X]$ and $A^p[Y]$ are algebraic $\ell^p$ Roe-type algebras associated to $X$ and $Y$ respectively, and there is an isometric isomorphism between $A^p[X]$ and $A^p[Y]$, then $X$ and $Y$ are coarsely equivalent. The same conclusion holds if $\mathbb{C}^p_s[X]$ and $\mathbb{C}^p_s[Y]$ are isometrically isomorphic.

The converse is true for $U\mathbb{C}^p[X]$, $\mathbb{C}^p[X]$, and $\mathbb{C}^p_s[X]$.
\end{enumerate}
\end{cor}

Note that we actually do not need to assume property A or finite decomposition complexity in the theorem as long as we exclude the case $p=2$. 
Isometric isomorphisms between $\ell^p$ uniform Roe algebras are spatially implemented by invertible isometries between the underlying $\ell^p$ spaces (see Lemma~\ref{spatially implemented}). Moreover, we have Lamperti's theorem from \cite{Lamp} that describes such invertible isometries when $p\in[1,\infty)\setminus\{2\}$. In particular, Lamperti's theorem provides a bijective map between the underlying metric spaces, and this map turns out to be a coarse equivalence. Another consequence of Lamperti's theorem is that such isometric isomorphisms map matrix units to matrix units up to multiplying by a scalar with absolute value one, which makes the arguments in the $p\neq 2$ case slightly simpler than in the $p=2$ case (e.g. in the proofs of Lemma \ref{Lem2} and Lemma \ref{Lem2s}). 
The same is true for isometric isomorphisms between general $\ell^p$ Roe-type algebras, except that Lamperti's theorem may provide a bijective map between $X\times\mathbb{N}$ and $Y\times\mathbb{N}$, from which one can obtain a coarse equivalence between $X$ and $Y$.
In the $C^\ast$-algebraic setting as in \cite{MR3116573, WW}, Lamperti's theorem is inapplicable, and property A and finite decomposition complexity played an essential role in producing a (bijective) coarse equivalence between the metric spaces.

Let us briefly describe how this paper is organized. In section 2, we consider isometric isomorphisms between $\ell^p$ uniform Roe algebras, while in section 3, we consider stable isometric isomorphisms between $\ell^p$ uniform Roe algebras and isometric isomorphisms between other $\ell^p$ Roe-type algebras. In fact, the arguments in section 3 are very similar to those in section 2, and require just a little more work, mainly because Lamperti's theorem gives us a map from $X\times\mathbb{N}$ to $Y\times\mathbb{N}$ in that case instead of a map from $X$ to $Y$.

\section{Isometrically isomorphic $\ell^p$ uniform Roe algebras}

In this section, we consider isometric isomorphisms between $\ell^p$ uniform Roe algebras associated to metric spaces with bounded geometry, and show that when $p\in[1,\infty)\setminus\{2\}$, these isometric isomorphisms give rise to bijective coarse equivalences between the underlying metric spaces.

We begin by noting that any isometric isomorphism between $\ell^p$ uniform Roe algebras must be spatially implemented by an invertible isometry.

\begin{lem}\label{spatially implemented}
Let $X$ and $Y$ be metric spaces with bounded geometry, and let $p\in[1,\infty)$. If $\phi:B^p_u(X)\rightarrow B^p_u(Y)$ is an algebra isomorphism between $\ell^p$ uniform Roe algebras, then there exists a bounded linear bijection $U:\ell^p(X)\rightarrow\ell^p(Y)$ such that $\phi(T)=UTU^{-1}$ for all $T\in B^p_u(X)$. In particular, $\phi$ is continuous.

Moreover, if $\phi$ is also isometric, then $U$ is an invertible isometry.
\end{lem}

\begin{proof}
Fix $x_0\in X$, and consider the rank one idempotent operator $\delta_{x_0}\otimes\delta_{x_0}\in\ell^p(X)\otimes\ell^p(X)^\ast$. Note that this operator has propagation zero so it belongs to $B^p_u(X)$. Now $\phi(\delta_{x_0}\otimes\delta_{x_0})$ is also a rank one idempotent operator, so $\phi(\delta_{x_0}\otimes\delta_{x_0})=f\otimes\sigma$ for some unit vector $f\in\ell^p(Y)$ and $\sigma\in\ell^p(Y)^\ast$ with $\sigma(f)=1$. If $\phi$ is isometric, then we also have $||\sigma||=1$.

Now note that $\xi\otimes\delta_{x_0}\in B^p_u(X)$ for all $\xi\in\ell^p(X)$. Moreover, we have
\[ \phi(\xi\otimes\delta_{x_0}) = \phi(\xi\otimes\delta_{x_0})\phi(\delta_{x_0}\otimes\delta_{x_0}) = \phi(\xi\otimes\delta_{x_0})f\otimes\sigma. \]
Define $U:\ell^p(X)\rightarrow\ell^p(Y)$ by $U\xi=\phi(\xi\otimes\delta_{x_0})f$. Then $U^{-1}\eta=\phi^{-1}(\eta\otimes\sigma)\delta_{x_0}$ for $\eta\in\ell^p(Y)$. Moreover, if $\phi$ is an isometry, then $U$ is an isometry. For any $T\in B^p_u(X)$ and $\xi, \sigma \in\ell^p(X)$, we have
\[ \phi(T)U\xi\otimes\sigma = \phi(T(\xi\otimes\delta_{x_0}))f\otimes\sigma = \phi(T\xi\otimes\delta_{x_0})f\otimes\sigma = UT\xi\otimes\sigma, \]
showing that $\phi(T)=UTU^{-1}$.

To see that $U$ is bounded, first note that $\delta_y\circ U$ is bounded for all $y\in Y$. Indeed, since $\eta\otimes\delta_y\in B^p_u(Y)$ for any $\eta\in\ell^p(Y)$ and $y\in Y$, there exists $T\in B^p_u(X)$ such that $UTU^{-1}=\phi(T)=\eta\otimes\delta_y$, and
\[ T=U^{-1}(\eta\otimes\delta_y)U=U^{-1}\eta\otimes(\delta_y\circ U). \]
Since $T$ is bounded, it follows that $\delta_y\circ U$ is bounded. Now suppose that $\xi_n\rightarrow 0$ in $\ell^p(X)$ and $U\xi_n\rightarrow\eta$ for some $\eta\in\ell^p(Y)$. Fix $y_0\in Y$. Then $(\delta_{y_0}\otimes\delta_y)U\xi_n\rightarrow(\delta_{y_0}\otimes\delta_y)\eta$ in $\ell^p(Y)$. On the other hand, 
\[\delta_{y_0}\otimes(\delta_y\circ U)=(\delta_{y_0}\otimes\delta_y)U=US\] for some $S\in B^p_u(X)$. 
Since $\delta_y\circ U$ is bounded, so is $US$, and thus $(\delta_{y_0}\otimes\delta_y)U\xi_n=US\xi_n\rightarrow 0$. Hence $(\delta_{y_0}\otimes\delta_y)\eta=0$ for all $y\in Y$, so $\eta=0$. By the Closed Graph Theorem, $U$ is bounded.
\end{proof}

\begin{rem} \label{rankone} \leavevmode
\begin{enumerate}
\item When $p\in(1,\infty)$, Lemma \ref{spatially implemented} is a special case of \cite[Corollary 3.2]{Cher} (or the theorem in \cite[Section 1.7.15]{Pal}) as $B^p_u(X)$ contains all finite rank operators on $\ell^p(X)$.
However, when $p=1$ and $X$ is infinite, $B^1_u(X)$ does not necessarily contain all finite rank operators on $\ell^1(X)$. For example, consider $X=\mathbb{N}$ with the natural metric, and consider the rank one idempotent operator on $\ell^1(\mathbb{N})$ given by $T\xi=(\sum_n\xi_n)\delta_0$. Suppose $S\in\mathbb{C}_u^1[\mathbb{N}]$ has propagation $R$. Then $S_{0,n}=0$ for all $n>R$ so $||(T-S)\delta_n||_{\ell^1(\mathbb{N})}\geq 1$ for all $n>R$, and $T$ is not a norm-limit of operators with finite propagation, i.e., $T\notin B^1_u(\mathbb{N})$. The proof we present above is a slight modification of that in \cite[Section 1.7.15]{Pal} so as to include the case $p=1$.
\item If $p,q\in(1,\infty)$ satisfy $\frac{1}{p}+\frac{1}{q}=1$, then the map $(T_{xy})\mapsto(\overline{T_{yx}})$ defines an isometric anti-isomorphism between $B^p_u(X)$ and $B^q_u(X)$. However, if $X$ is infinite, then as a consequence of Lemma \ref{spatially implemented}, there is no algebra isomorphism between $B^p_u(X)$ and $B^q_u(X)$ whenever $1\leq p<q<\infty$ since there is no bounded linear bijection between $\ell^p(X)$ and $\ell^q(X)$ in this case as a classical theorem by Pitt \cite{Pitt} (or see \cite[Theorem 2.1.4]{AlbKal}) tells us that any bounded linear operator from $\ell^q(X)$ to $\ell^p(X)$ must be compact.
\end{enumerate}
\end{rem}

A key ingredient for us is Lamperti's theorem in \cite{Lamp}, which we state in the following form.
\begin{prop} \label{Lam}
Let $X$ and $Y$ be metric spaces with bounded geometry, and let $p\in[1,\infty)\setminus\{2\}$. If $U:\ell^p(X)\rightarrow\ell^p(Y)$ is an invertible isometry, then there exists a function $h:Y\rightarrow\mathbb{T}$ and an invertible function $g:X\rightarrow Y$ such that $(U\xi)(y)=h(y)\xi(g^{-1}(y))$ for all $\xi\in\ell^p(X)$ and $y\in Y$.
\end{prop}

Note that for $x\in X$ and $y\in Y$, we have \[ U\delta_x=h(g(x))\delta_{g(x)} \; \text{and} \; U^{-1}\delta_y=\overline{h(y)}\delta_{g^{-1}(y)}. \]
In particular, $|\langle U\delta_x,\delta_{g(x)} \rangle|=1=|\langle \delta_{g^{-1}(y)},U^{-1}\delta_y \rangle|$ for $x\in X$ and $y\in Y$.

Also note that if $\phi:B^p_u(X)\rightarrow B^p_u(Y)$ is given by $\phi(T)=UTU^{-1}$, and $e_{x_1,x_2}$ is a matrix unit, then
\[ \phi(e_{x_1,x_2})=Ue_{x_1,x_2}U^{-1}=h(g(x_1))\overline{h(g(x_2))}e_{g(x_1),g(x_2)}. \]

\begin{rem}
In the $C^\ast$-algebraic setting considered in \cite{MR3116573}, Lamperti's theorem does not apply, so instead property A is used in the form of the metric sparsification property, which in turn allows one to show that if $U:\ell^2(X)\rightarrow\ell^2(Y)$ is the unitary operator implementing the isomorphism $\phi$, then there exists $c>0$ such that for each $x\in X$ there exists $f(x)\in Y$ with $|\langle U\delta_x,\delta_{f(x)} \rangle|\geq c$, and similarly for $U^\ast$. 
\end{rem}

The next lemma shows that the function $g$ in Proposition \ref{Lam}, as well as $g^{-1}$, is a bijective coarse equivalence.

\begin{lem} \label{Lem2}
Let $g$ be as in Proposition \ref{Lam}.
For all $R\geq 0$ there exists $S\geq 0$ such that if $x_1,x_2\in X$ are such that $d(x_1,x_2)\leq R$, then $d(g(x_1),g(x_2))\leq S$. Also, if $y_1,y_2\in Y$ are such that $d(y_1,y_2)\leq R$, then $d(g^{-1}(y_1),g^{-1}(y_2))\leq S$.
\end{lem}

\begin{proof}
We prove only the first statement as the same argument holds with the roles of $X$ and $Y$ reversed and with $g^{-1}$ instead of $g$.

Assume for contradiction that it is false. Then there exists $R\geq 0$ and sequences $(x_1^n),(x_2^n)$ such that $d(x_1^n,x_2^n)\leq R$ for all $n$, and $d(g(x_1^n),g(x_2^n))\rightarrow\infty$ as $n\rightarrow\infty$. 
We may assume that $(x_1^n,x_2^n)\neq(x_1^m,x_2^m)$ if $n\neq m$.
Thus $\sum_{n\in\mathbb{N}}e_{x_1^n,x_2^n}$ converges strongly to a bounded operator on $\ell^p(X)$ that is moreover in $\mathbb{C}_u^p[X]\subseteq B^p_u(X)$. Hence the sum \[\sum_{n\in\mathbb{N}}\phi(e_{x_1^n,x_2^n})=\sum_{n\in\mathbb{N}}h(g(x_1^n))\overline{h(g(x_2^n))}e_{g(x_1^n),g(x_2^n)}\] converges strongly to an operator in $B^p_u(Y)$.
But this contradicts the assumption that $d(g(x_1^n),g(x_2^n))\rightarrow\infty$ as $n\rightarrow\infty$.
\end{proof}

The next theorem generalizes the non-$K$-theoretic part of \cite[Theorem 4.10]{LL}. The implication (4) $\Rightarrow$ (1) is the $p\neq 2$ analog of \cite[Corollary~1.16]{WW}, which together with \cite[Lemma~8]{MR0043392} tells us that in the $p=2$ case, an isometric isomorphism between the uniform Roe algebras yields a bijective coarse equivalence if the metric spaces have finite decomposition complexity.

\begin{thm} \label{thm1}
Let $X$ and $Y$ be metric spaces with bounded geometry. The following are equivalent:
\begin{enumerate}
\item $X$ and $Y$ are bijectively coarsely equivalent.
\item For every $p\in[1,\infty)$, there is an isometric isomorphism $\phi:B^p_u(X)\rightarrow B^p_u(Y)$ such that $\phi(\ell^\infty(X))=\ell^\infty(Y)$.
\item $B^p_u(X)$ and $B^p_u(Y)$ are isometrically isomorphic for every $p\in[1,\infty)$.
\item $B^p_u(X)$ and $B^p_u(Y)$ are isometrically isomorphic for some $p\in[1,\infty)\setminus\{2\}$.

\item For some $p\in[1,\infty)$, there is an isometric isomorphism $\phi:B^p_u(X)\rightarrow B^p_u(Y)$ such that $\phi(\ell^\infty(X))=\ell^\infty(Y)$.
\end{enumerate}
\end{thm}

\begin{proof}
(1) $\Rightarrow$ (2): 
The proof we present here is based on the proof of \cite[Theorem 4]{BNW} and \cite[Proposition~2.3]{LL}.

Let $f:X\rightarrow Y$ be a bijective coarse equivalence, and consider the invertible isometry $U:\ell^p(X)\rightarrow\ell^p(Y)$ satisfying $U\delta_x=\delta_{f(x)}$ for all $x\in X$. Define $\phi:B(\ell^p(X))\rightarrow B(\ell^p(Y))$ by $\phi(T)=UTU^{-1}$. We show that $\phi$ maps $B^p_u(X)$ into $B^p_u(Y)$.

Note that for all $x_1,x_2\in X$, we have
\[ \langle \phi(T)\delta_{f(x_2)},\delta_{f(x_1)} \rangle=\langle T\delta_{x_2},\delta_{x_1} \rangle. \]
Suppose that $\langle T\delta_{x'},\delta_x \rangle=0$ whenever $d(x',x)>R$. There exists $S>0$ such that $d(f(x_1),f(x_2))<S$ whenever $d(x_1,x_2)<R$. Thus $\langle \phi(T)\delta_{f(x')},\delta_{f(x)} \rangle=0$ whenever $d(f(x'),f(x))>S$.
Since $\phi$ is continuous, it maps $B^p_u(X)$ into $B^p_u(Y)$.

Reversing the roles of $X$ and $Y$, one sees that the homomorphism given by $\psi(T)=U^{-1}TU$ maps $B^p_u(Y)$ into $B^p_u(X)$, and is the inverse of $\phi$. Hence $\phi:B^p_u(X)\rightarrow B^p_u(Y)$ is an isometric isomorphism. Moreover, the definition of $\phi$ shows that it maps $\ell^\infty(X)$ onto $\ell^\infty(Y)$.

(2) $\Rightarrow$ (3) $\Rightarrow$ (4) is trivial.

(4) $\Rightarrow$ (1): If $B^p_u(X)$ and $B^p_u(Y)$ are isometrically isomorphic for some $p\in[1,\infty)\setminus\{2\}$, then by Lemma \ref{spatially implemented}, we have an invertible isometry $U:\ell^p(X)\rightarrow\ell^p(Y)$.
Let $g:X\rightarrow Y$ be as in Proposition \ref{Lam}. Then Lemma \ref{Lem2} implies that both $g$ and $g^{-1}$ are uniformly expansive, so $g:X\rightarrow Y$ is a bijective coarse equivalence.

(2) $\Rightarrow$ (5): It is clear.

(5) $\Rightarrow$ (1): If $p=2$, we are done by \cite[Lemma~8]{MR0043392} and \cite[Corollay~1.16]{WW}. If $p \neq 2$, it is clear that (5) $\Rightarrow$ (4). 
\end{proof}

Recall from \cite[Definition~1.11]{WW} that a bounded geometry metric space $X$ is \emph{bijectively rigid} if whenever there is a coarse equivalence $f:X\rightarrow Y$ to another bounded geometry metric space $Y$, then there is a bijective coarse equivalence $f':X\rightarrow Y$. It can be deduced from the proof of \cite[Theorem 1.1]{Whyte} that every uniformly discrete, non-amenable bounded geometry metric space is bijectively rigid. It is elementary to see that $\Z$ is also bijectively rigid. On the other hand, locally finite groups and certain lamplighter groups are not (see \cite{LL} and \cite{MR2730576}). 

\begin{cor} \label{cor1}
Let $X$ and $Y$ be bijectively rigid metric spaces with bounded geometry. The following are equivalent:
\begin{enumerate}
\item $X$ and $Y$ are coarsely equivalent.
\item For every $p\in[1,\infty)$, there is an isometric isomorphism $\phi:B^p_u(X)\rightarrow B^p_u(Y)$ such that $\phi(\ell^\infty(X))=\ell^\infty(Y)$.
\item $B^p_u(X)$ and $B^p_u(Y)$ are isometrically isomorphic for every $p\in[1,\infty)$.
\item $B^p_u(X)$ and $B^p_u(Y)$ are isometrically isomorphic for some $p\in[1,\infty)\setminus\{2\}$.
\item For some $p\in[1,\infty)$, there is an isometric isomorphism $\phi:B^p_u(X)\rightarrow B^p_u(Y)$ such that $\phi(\ell^\infty(X))=\ell^\infty(Y)$.
\end{enumerate}
\end{cor}

For general metric spaces with bounded geometry, coarse equivalence only corresponds to stable isometric isomorphism of their $\ell^p$ uniform Roe algebras, as we shall see in the next section.

\begin{rem} We make a few remarks about the $p=2$ case.
\begin{enumerate}
\item The proof of \cite[Theorem 4]{BNW} or \cite[Proposition~2.3]{LL} shows that a bijective coarse equivalence between $X$ and $Y$ yields a $*$-isomorphism $\phi$ between the uniform Roe $C^\ast$-algebras $C^\ast_u(X)$ and $C^\ast_u(Y)$ such that $\phi(\ell^\infty(X))=\ell^\infty(Y)$.
\item For the converse, \cite[Theorem 4.1]{MR3116573} says that if $X$ and $Y$ are assumed to have property A, then a $*$-isomorphism between their uniform Roe $C^\ast$-algebras yields a coarse equivalence between $X$ and $Y$.
\item If $X$ and $Y$ are countable locally finite groups equipped with proper left-invariant metrics, then a $*$-isomorphism between their uniform Roe $C^\ast$-algebras yields a bijective coarse equivalence between $X$ and $Y$ by \cite[Theorem 4.10]{LL}.
\end{enumerate}
\end{rem}

For the uncompleted algebras $\mathbb{C}_u^p[X]$, we have the following result.

\begin{cor} \label{cor2}
Let $X$ and $Y$ be metric spaces with bounded geometry, and let $p\in[1,\infty)$. Then $\mathbb{C}_u^p[X]$ and $\mathbb{C}_u^p[Y]$ are isometrically isomorphic if and only if $X$ and $Y$ are bijectively coarsely equivalent.
\end{cor}
\begin{proof}
If $X$ and $Y$ are bijectively coarsely equivalent, then the proof of (1) $\Rightarrow$ (2) in Theorem~\ref{thm1} actually implies that $\mathbb{C}_u^p[X]$ and $\mathbb{C}_u^p[Y]$ are isometrically isomorphic for every $p\in[1,\infty)$.

Now we assume that $\mathbb{C}_u^p[X]$ and $\mathbb{C}_u^p[Y]$ are isometrically isomorphic for some $p\in[1,\infty)$. If $p\neq 2$, we can apply Theorem~\ref{thm1}. If $p=2$, then \cite[Lemma~8]{MR0043392} and the proof in \cite{WW} or \cite[Theorem~8.1]{BF18} actually imply that $X$ and $Y$ are bijectively coarsely equivalent without assuming FDC \footnote{We could compare this with Theorem~1.4 in \cite{MR3116573}.}.
\end{proof}

\section{Stably isometrically isomorphic $\ell^p$ uniform Roe algebras}

In this section, we consider stable isometric isomorphisms between $\ell^p$ uniform Roe algebras associated to metric spaces with bounded geometry, and show that when $p\in [1,\infty)\setminus\{2\}$, these stable isometric isomorphisms give rise to (not necessarily bijective) coarse equivalences between the underlying metric spaces. 
We also consider general $\ell^p$ Roe-type algebras.
The ingredients are essentially the same as in the previous section.

We will need to consider tensor products of $L^p$ operator algebras so we begin by making this notion precise.
Details can be found in \cite[Chapter 7]{DF} and \cite[Theorem 2.16]{Phil12}.

For $p\in[1,\infty)$, there is a tensor product of $L^p$ spaces with $\sigma$-finite measures such that we have a canonical isometric isomorphism $L^p(X,\mu)\otimes L^p(Y,\nu)\cong L^p(X\times Y,\mu\times\nu)$, which identifies, for every $\xi\in L^p(X,\mu)$ and $\eta\in L^p(Y,\nu)$, the element $\xi\otimes\eta$ with the function $(x,y)\mapsto\xi(x)\eta(y)$ on $X\times Y$. Moreover, we have the following properties:
\begin{itemize}
\item Under the identification above, the linear span of all $\xi\otimes\eta$ is dense in $L^p(X\times Y,\mu\times\nu)$.
\item $||\xi\otimes\eta||_p=||\xi||_p||\eta||_p$ for all $\xi\in L^p(X,\mu)$ and $\eta\in L^p(Y,\nu)$.
\item The tensor product is commutative and associative.
\item If $a\in B(L^p(X_1,\mu_1),L^p(X_2,\mu_2))$ and $b\in B(L^p(Y_1,\nu_1),L^p(Y_2,\nu_2))$, then there exists a unique \[c\in B(L^p(X_1\times Y_1,\mu_1\times\nu_1),L^p(X_2\times Y_2,\mu_2\times\nu_2))\] such that under the identification above, $c(\xi\otimes\eta)=a(\xi)\otimes b(\eta)$ for all $\xi\in L^p(X_1,\mu_1)$ and $\eta\in L^p(Y_1,\nu_1)$. We will denote this operator by $a\otimes b$. Moreover, $||a\otimes b||=||a|| ||b||$.
\item The tensor product of operators is associative, bilinear, and satisfies $(a_1\otimes b_1)(a_2\otimes b_2)=a_1a_2\otimes b_1b_2$.
\end{itemize}
If $A\subseteq B(L^p(X,\mu))$ and $B\subseteq B(L^p(Y,\nu))$ are norm-closed subalgebras, we then define $A\otimes B\subseteq B(L^p(X\times Y,\mu\times\nu))$ to be the closed linear span of all $a\otimes b$ with $a\in A$ and $b\in B$.

We will write $\overline{M}^p_\infty$ for $\overline{\bigcup_{n\in\mathbb{N}}M_n(\mathbb{C})}\subset B(\ell^p(\N))$. When $p\in(1,\infty)$, $\overline{M}^p_\infty$ is equal to the algebra $K(\ell^p(\N))$ of compact operators on $\ell^p(\N)$ \cite[Corollary~1.9]{Phil13}, but when $p=1$, there is a rank one operator (in fact, the operator in Remark \ref{rankone}(1)) that is not in $\overline{M}^1_\infty$ \cite[Example 1.10]{Phil13}, so $\overline{M}^1_\infty$ is strictly contained in $K(\ell^1(\N))$. 
We will regard elements of $B^p_u(X)\otimes \overline{M}^p_\infty$ as bounded operators on $\ell^p(X\times\mathbb{N})$.

The following lemma is analogous to Lemma \ref{spatially implemented}, and is proved in the same way with the obvious modifications.

\begin{lem} \label{spatially implemented 2}
Let $X$ and $Y$ be metric spaces with bounded geometry, and let $p\in[1,\infty)$. If $\phi:B^p_u(X)\otimes \overline{M}^p_\infty\rightarrow B^p_u(Y)\otimes \overline{M}^p_\infty$ is an algebra isomorphism, then there exists a bounded linear bijection $U:\ell^p(X\times\mathbb{N})\rightarrow\ell^p(Y\times\mathbb{N})$ such that $\phi(T)=UTU^{-1}$ for all $T\in B^p_u(X)\otimes \overline{M}^p_\infty$. In particular, $\phi$ is continuous.

Moreover, if $\phi$ is also isometric, then $U$ is an invertible isometry.
\end{lem}

We now use Lamperti's theorem in the following form.
\begin{prop} \label{Lam2}
Let $X$ and $Y$ be metric spaces with bounded geometry, and let $p\in[1,\infty)\setminus\{2\}$. If $U:\ell^p(X\times\mathbb{N})\rightarrow\ell^p(Y\times\mathbb{N})$ is an invertible isometry, then there exists a function $h:Y\times\mathbb{N}\rightarrow\mathbb{T}$ and an invertible function $g:X\times\mathbb{N}\rightarrow Y\times\mathbb{N}$ such that $(U\xi)(y,m)=h(y,m)\xi(g^{-1}(y,m))$ for all $\xi\in\ell^p(X\times\mathbb{N})$, $y\in Y$, and $m\in\mathbb{N}$.
\end{prop}

Note that for $x\in X$, $y\in Y$, and $n,m\in\mathbb{N}$, we have \[ U\delta_{x,n}=h(g(x,n))\delta_{g(x,n)} \; \text{and} \; U^{-1}\delta_{y,m}=\overline{h(y,m)}\delta_{g^{-1}(y,m)}. \]
Let $\pi_X:X\times\mathbb{N}\rightarrow X$ and $\pi_Y:Y\times\mathbb{N}\rightarrow Y$ denote the respective coordinate projections. Then consider the maps $f:X\rightarrow Y$ and $f':Y\rightarrow X$ given by $f(x)=\pi_Y(g(x,0))$ and $f'(y)=\pi_X(g^{-1}(y,0))$.

The next lemma shows that $f$ and $f'$ are uniformly expansive, and will also be used in the proof of Theorem \ref{thm2} to show that $f\circ f'$ and $f'\circ f$ are close to the identity.
The proof is essentially the same as that of Lemma \ref{Lem2} and is modeled after \cite[Lemma 4.5]{MR3116573}.

\begin{lem} \label{Lem2s}
Let $g$ be as in Proposition \ref{Lam2}. For all $R\geq 0$, there exists $S\geq 0$ such that if $x,x'\in X$ satisfy $d(x,x')\leq R$ and $y,y'\in Y$ satisfy $|\langle U\delta_{x,n},\delta_{y,m} \rangle|=1=|\langle U\delta_{x',n'},\delta_{y',m'} \rangle|$ for some $m,n,m',n'\in\mathbb{N}$, then $d(y,y')\leq S$.

The same properties hold with the roles of $X$ and $Y$ reversed, and with $U$ replaced by $U^{-1}$.

In particular, if $f(x)=\pi_Y(g(x,0))$ and $f'(y)=\pi_X(g^{-1}(y,0))$, then 
\begin{enumerate}
\item for all $R\geq 0$, there exists $S\geq 0$ such that if $x,x'\in X$ are such that $d(x,x')\leq R$, then $d(f(x),f(x'))\leq S$;
\item for all $R\geq 0$, there exists $S\geq 0$ such that if $y,y'\in Y$ are such that $d(y,y')\leq R$, then $d(f'(y),f'(y'))\leq S$.
\end{enumerate}
i.e., $f$ and $f'$ are uniformly expansive.
\end{lem}

\begin{proof}
Assume for contradiction that it is false. Then there exist sequences $(x_k),(x'_k),(y_k),(y'_k),(m_k),(m'_k),(n_k),(n'_k)$ such that $d(x_k,x'_k)\leq R$ for each $k$, \[|\langle U\delta_{x_k,n_k},\delta_{y_k,m_k} \rangle|=1=|U\langle \delta_{x'_k,n'_k},\delta_{y'_k,m'_k} \rangle|,\] and $d(y_k,y'_k)\rightarrow\infty$ as $k\rightarrow\infty$. 
Now at least one of the sequences $(y_k)$ or $(y'_k)$ must have a subsequence tending to infinity in $Y$ so without loss of generality, we assume that $(y_k)$ itself tends to infinity in $Y$. It follows that the sequence $(\delta_{y_k,m_k})$ of unit vectors in $\ell^p(Y\times\mathbb{N})$ tends weakly to zero. Thus the sequence $(\delta_{x_k,n_k})$ must eventually leave any norm-compact subset of $\ell^p(X\times\mathbb{N})$, and $(x_k)$ must tend to infinity in $X$. Passing to another subsequence, we may assume that $d(x_k,x_{k+1})>2R$ for all $k$. Since $d(x_k,x'_k)\leq R$, the elements $x_k$ and $x'_k$ are all distinct, and the sum $\sum_{k\in\mathbb{N}}e_{(x_k,n_k),(x'_k,n'_k)}$ converges strongly to a bounded operator on $\ell^p(X\times\mathbb{N})$ that is moreover in $B^p_u(X)\otimes \overline{M}^p_\infty$. Hence the sum \[\sum_{k\in\mathbb{N}}\phi(e_{(x_k,n_k),(x'_k,n'_k)})=\sum_{k\in\mathbb{N}}h(y_k,m_k)\overline{h(y'_k,m'_k)}e_{(y_k,m_k),(y'_k,m'_k)}\] converges strongly to a bounded operator in $B^p_u(Y)\otimes \overline{M}^p_\infty$.
In particular, the sum $\sum_{k\in\mathbb{N}}e_{y_k,y'_k}$ converges strongly to an operator in $B^p_u(Y)$.
But this contradicts the property that $d(y_k,y'_k)\rightarrow\infty$ as $k\rightarrow\infty$.

The same argument works with the roles of $X$ and $Y$ reversed, and with $U$ replaced by $U^{-1}$.

Statements (1) and (2) follow from the observation that
\[|\langle U\delta_{x,0},\delta_{f(x),m} \rangle|=1=|\langle U\delta_{x',0},\delta_{f(x'),n} \rangle|\] and 
\[|\langle U^{-1}\delta_{y,0},\delta_{f'(y),m'} \rangle|=1=|\langle U^{-1}\delta_{y',0},\delta_{f'(y'),n'} \rangle|\] for some $n,m,n',m'\in\mathbb{N}$.
\end{proof}

The implication (1) $\Rightarrow$ (2) in the following theorem generalizes \cite[Theorem 4]{BNW} to all $p\in[1,\infty)$, while the implication (4) $\Rightarrow$ (1) is the $p\neq 2$ analog of \cite[Corollary 6.2]{MR3116573}, which together with \cite[Lemma~8]{MR0043392} tells us that in the $p=2$ case, stable isometric isomorphisms between uniform Roe algebras yield coarse equivalences if the metric spaces have property A.

\begin{thm} \label{thm2}
Let $X$ and $Y$ be metric spaces with bounded geometry. The following are equivalent:
\begin{enumerate}
\item $X$ and $Y$ are coarsely equivalent.
\item For every $p\in[1,\infty)$, there is an isometric isomorphism $\phi:B^p_u(X)\otimes \overline{M}^p_\infty\rightarrow B^p_u(Y)\otimes \overline{M}^p_\infty$ such that $\phi(\ell^\infty(X)\otimes C_0(\mathbb{N}))=\ell^\infty(Y)\otimes C_0(\mathbb{N})$.
\item $B^p_u(X)\otimes \overline{M}^p_\infty$ and $B^p_u(Y)\otimes \overline{M}^p_\infty$ are isometrically isomorphic for every $p\in[1,\infty)$.
\item $B^p_u(X)\otimes \overline{M}^p_\infty$ and $B^p_u(Y)\otimes \overline{M}^p_\infty$ are isometrically isomorphic for some $p\in[1,\infty)\setminus\{2\}$.

\item For some $p\in[1,\infty)$, there is an isometric isomorphism $\phi:B^p_u(X)\otimes \overline{M}^p_\infty\rightarrow B^p_u(Y)\otimes \overline{M}^p_\infty$ such that $\phi(\ell^\infty(X)\otimes C_0(\mathbb{N}))=\ell^\infty(Y)\otimes C_0(\mathbb{N})$.
\end{enumerate}
\end{thm}

\begin{proof}
(1) $\Rightarrow$ (2):
The proof we present here is essentially the same as the proof of \cite[Theorem 4]{BNW}.

We first assume that there is a surjective coarse equivalence $f:X\rightarrow Y$. Since $f$ is a coarse equivalence, there exists $R>0$ such that for each $y\in Y$, the preimage $f^{-1}(y)$ lies in some $R$-ball in $X$. Then since $X$ has bounded geometry, there exists $N$ such that the cardinality of $f^{-1}(y)$ is at most $N$ for each $y\in Y$. Define $N(y)$ to be the cardinality of $f^{-1}(y)$. For each $y\in Y$, enumerating the points of $f^{-1}(y)$ gives a bijection between $f^{-1}(y)$ and $\{1,\ldots,N(y)\}\subseteq\{1,\ldots,N\}$. We therefore obtain an identification of $X$ with a subset of $Y\times\{1,\ldots,N\}$. Let $\pi$ denote the corresponding projection from $X$ to $\{1,\ldots,N\}$, so that the identification is given by $x\mapsto(f(x),\pi(x))$.

Define a map $\phi:X\times\mathbb{N}\rightarrow Y\times\mathbb{N}$ by $\phi(x,j)=(f(x),\pi(x)+jN(f(x)))$. Since for each $y\in Y$ there is exactly one $x\in X$ satisfying $f(x)=y$ and $\pi(x)=i$ for $i=1,\ldots,N(y)$, the map $\phi$ is a bijection, which gives rise to an invertible isometry from $\ell^p(X\times\mathbb{N})$ to $\ell^p(Y\times\mathbb{N})$, and thus an isometric isomorphism $\Phi$ from $B(\ell^p(X\times\mathbb{N}))$ to $B(\ell^p(Y\times\mathbb{N}))$. We will show that $\Phi$ maps $B^p_u(X)\otimes \overline{M}^p_\infty$ into $B^p_u(Y)\otimes \overline{M}^p_\infty$, while $\Phi^{-1}$ maps $B^p_u(Y)\otimes \overline{M}^p_\infty$ into $B^p_u(X)\otimes \overline{M}^p_\infty$, and hence $\Phi$ restricts to an isometric isomorphism between $B^p_u(X)\otimes \overline{M}^p_\infty$ and $B^p_u(Y)\otimes \overline{M}^p_\infty$.

We consider the dense subalgebra of $B^p_u(X)\otimes \overline{M}^p_\infty$ consisting of sums of elementary tensors of the form $T\otimes e_{j,j'}$, where $T$ is a finite propagation operator on $\ell^p(X)$ and $e_{j,j'}$ is a matrix unit. It suffices to show that such elementary tensors are mapped by $\Phi$ into $B^p_u(Y)\otimes \overline{M}^p_\infty$.

Partition $X$ as $X=\bigcup_{n=1,\ldots,N,i=1,\ldots,n}X_{n,i}$, where \[ X_{n,i}=\{ x\in X:N(f(x))=n,\pi(x)=i \}. \]
We can write $T$ as \[ T=\sum_{\substack{n,n'\leq N,\\ i\leq n,i'\leq n'}}P_{n,i}TP_{n',i'}, \] where $P_{n,i}$ denotes the projection of $\ell^p(X)$ onto $\ell^p(X_{n,i})$.
Note that for each $i$, the restriction of $f$ to $X_i=\bigcup_{n\geq i}X_{n,i}$ is injective, and let $V_i:\ell^p(X_i)\rightarrow\ell^p(Y)$ denote the corresponding isometry. Also, let $V_i^\dagger:\ell^p(Y)\rightarrow\ell^p(X_i)$ denote the reverse (in the terminology of Phillips \cite[Definition 6.13]{Phil12}).

Now fix $n,n',i,i'$, and let $S=P_{n,i}TP_{n',i'}$. Then $\Phi(S\otimes e_{j,j'})=V_iSV_{i'}^\dagger\otimes e_{i+nj,i'+n'j'}$. Since $f$ is a coarse equivalence, the operator $V_iSV_{i'}^\dagger$ has finite propagation, and hence $V_iSV_{i'}^\dagger\otimes e_{i+nj,i'+n'j'}$ lies in $B^p_u(Y)\otimes \overline{M}^p_\infty$. Since this holds for each $n,n',i,i'$, we conclude that $\Phi(T\otimes e_{j,j'})\in B^p_u(Y)\otimes \overline{M}^p_\infty$.

Showing that $\Phi^{-1}$ maps $B^p_u(Y)\otimes \overline{M}^p_\infty$ into $B^p_u(X)\otimes \overline{M}^p_\infty$ is done in a similar manner, and we omit the details, referring the reader to the proof of \cite[Theorem 4]{BNW}.

Finally, to remove the assumption that the coarse equivalence is surjective, observe that given any coarse equivalence $f:X\rightarrow Y$, there are surjective coarse equivalences from both $X$ and $Y$ to the image $f(X)$. Hence we have isometric isomorphisms \[ B^p_u(X)\otimes \overline{M}^p_\infty\cong B^p_u(f(X))\otimes \overline{M}^p_\infty\cong B^p_u(Y)\otimes \overline{M}^p_\infty. \]
From the definition of $\Phi$, one sees that this isometric isomorphism maps $\ell^\infty(X)\otimes C_0(\mathbb{N})$ onto $\ell^\infty(Y)\otimes C_0(\mathbb{N})$.

(2) $\Rightarrow$ (3) $\Rightarrow$ (4) is trivial.

(4) $\Rightarrow$ (1):
If $B^p_u(X)\otimes \overline{M}^p_\infty$ and $B^p_u(Y)\otimes \overline{M}^p_\infty$ are isometrically isomorphic for some $p\in[1,\infty)\setminus\{2\}$, then by Lemma \ref{spatially implemented 2}, we have an invertible isometry $U:\ell^p(X\times\mathbb{N})\rightarrow\ell^p(Y\times\mathbb{N})$.
Let $g$ be as in Proposition \ref{Lam2}. Then Lemma \ref{Lem2s} implies that $f=\pi_Y(g(-,0))$ and $f'=\pi_X(g^{-1}(-,0))$ are uniformly expansive. It remains to be shown that $f\circ f'$ and $f'\circ f$ are close to the identity.

For each $y\in Y$ we have 
\[ |\langle U\delta_{f'(y),0},\delta_{f(f'(y)),m} \rangle|=1=|\langle U\delta_{f'(y),n},\delta_{y,0} \rangle| \] 
for some $n,m\in\mathbb{N}$ so Lemma \ref{Lem2s} implies the existence of $S\geq 0$ (independently of $y$) such that $d(f(f'(y)),y)\leq S$, i.e., $f\circ f'$ is close to the identity. Similarly, for each $x\in X$ we have 
\[ |\langle U^{-1}\delta_{f(x),0},\delta_{f'(f(x)),m} \rangle|=1=|\langle U^{-1}\delta_{f(x),n},\delta_{x,0} \rangle| \] 
for some $n,m\in\mathbb{N}$ so $f'\circ f$ is close to the identity.

(2) $\Rightarrow$ (5): It is clear.

(5) $\Rightarrow$ (1): If $p=2$, we are done by \cite[Lemma~8]{MR0043392} and Proposition~\ref{p=2 case}, which is independent of Theorem~\ref{thm2}. If $p \neq 2$, it is clear that (5) $\Rightarrow$ (4). 
\end{proof}

Given a countable group $\Gamma$ and $p\in[1,\infty)$, we may represent elements of $\ell^\infty(\Gamma)$ as multiplication operators on $\ell^p(\Gamma)$, and consider the left translation action of $\Gamma$ on $\ell^\infty(\Gamma)$. Then one can define an $L^p$ reduced crossed product $\ell^\infty(\Gamma)\rtimes_{\lambda,p}\Gamma$ just as how one defines the reduced crossed product $C^\ast$-algebra (cf. \cite[Definition 4.1.4]{BO}). Then the proof of \cite[Proposition 5.1.3]{BO} shows that $\ell^\infty(\Gamma)\rtimes_{\lambda,p}\Gamma$ is isometrically isomorphic to $B^p_u(\Gamma)$.

For finitely generated groups, we may replace coarse equivalence by quasi-isometry in Theorem \ref{thm2}. Moreover, if $\Gamma$ and $\Lambda$ are non-amenable finitely generated groups, then Corollary \ref{cor1} applies, so if they are quasi-isometric, then we get isometric isomorphisms between their $\ell^p$ uniform Roe algebras instead of just stable isometric isomorphisms.

\begin{cor} \label{cor:grps}
Let $\Gamma$ and $\Lambda$ be finitely generated groups. The following are equivalent:
\begin{enumerate}
\item $\Gamma$ and $\Lambda$ are quasi-isometric.
\item For every $p\in[1,\infty)$, there is an isometric isomorphism \[\phi:(\ell^\infty(\Gamma)\rtimes_{\lambda,p}\Gamma)\otimes \overline{M}^p_\infty\rightarrow (\ell^\infty(\Lambda)\rtimes_{\lambda,p}\Lambda)\otimes \overline{M}^p_\infty\] such that $\phi(\ell^\infty(\Gamma)\otimes C_0(\mathbb{N}))=\ell^\infty(\Lambda)\otimes C_0(\mathbb{N})$.
\item $(\ell^\infty(\Gamma)\rtimes_{\lambda,p}\Gamma)\otimes \overline{M}^p_\infty$ and $(\ell^\infty(\Lambda)\rtimes_{\lambda,p}\Lambda)\otimes \overline{M}^p_\infty$ are isometrically isomorphic for every $p\in[1,\infty)$.
\item $(\ell^\infty(\Gamma)\rtimes_{\lambda,p}\Gamma)\otimes \overline{M}^p_\infty$ and $(\ell^\infty(\Lambda)\rtimes_{\lambda,p}\Lambda)\otimes \overline{M}^p_\infty$ are isometrically isomorphic for some $p\in[1,\infty)\setminus\{2\}$.
\item For some $p\in[1,\infty)$, there is an isometric isomorphism \[\phi:(\ell^\infty(\Gamma)\rtimes_{\lambda,p}\Gamma)\otimes \overline{M}^p_\infty\rightarrow (\ell^\infty(\Lambda)\rtimes_{\lambda,p}\Lambda)\otimes \overline{M}^p_\infty\] such that $\phi(\ell^\infty(\Gamma)\otimes C_0(\mathbb{N}))=\ell^\infty(\Lambda)\otimes C_0(\mathbb{N})$.
\end{enumerate}
If moreover $\Gamma$ and $\Lambda$ are non-amenable, then these are also equivalent to:
\begin{enumerate}[resume]
\item $\ell^\infty(\Gamma)\rtimes_{\lambda,p}\Gamma$ and $\ell^\infty(\Lambda)\rtimes_{\lambda,p}\Lambda$ are isometrically isomorphic for every $p\in[1,\infty)$.
\item $\ell^\infty(\Gamma)\rtimes_{\lambda,p}\Gamma$ and $\ell^\infty(\Lambda)\rtimes_{\lambda,p}\Lambda$ are isometrically isomorphic for some $p\in[1,\infty)\setminus\{2\}$.
\end{enumerate}
\end{cor}

In the $p=2$ case, isometric isomorphisms between the algebras yield quasi-isometries between the groups if the groups have property A (or equivalently, if they are $C^\ast$-exact) by \cite[Corollary 6.2]{MR3116573}, \cite[Corollary 6.3]{MR3116573}, and \cite[Lemma~8]{MR0043392}.

Instead of $B^p_u(X)\otimes \overline{M}^p_\infty$, one may consider $\ell^p$ Roe-type algebras, as was done in the $C^\ast$-algebraic setting in \cite{MR3116573}, and also $B^p_u(X)\otimes K(\ell^p(\N))$, which is not quite a $\ell^p$ Roe-type algebra and is different from $B^p_u(X)\otimes \overline{M}^p_\infty$ when $p=1$.

\begin{defn} \label{Roetype}
Let $X$ be a metric space with bounded geometry, let $p\in[1,\infty)$, and let $S$ be either $\mathbb{N}$ or $\{1,\ldots,n\}$ for some $n\in\mathbb{N}$. An operator $T=(T_{xy})_{x,y\in X}$ in $B(\ell^p(X,\ell^p(S)))$ is said to be locally compact if $T_{xy}\in K(\ell^p(S))$ for all $x,y\in X$.
A subalgebra $A^p[X]$ of $B(\ell^p(X,\ell^p(S)))$ is called an algebraic $\ell^p$ Roe-type algebra if
\begin{itemize}
\item it consists only of locally compact, finite propagation operators,
\item it contains all finite propagation operators $T=(T_{xy})_{x,y\in X}$ with the property that there exists $N\in\mathbb{N}$ such that the rank of $T_{xy}$ is at most $N$ for all $x,y\in X$.
\end{itemize}
An $\ell^p$ Roe-type algebra $A^p(X)$ is the operator norm closure of some algebraic $\ell^p$ Roe-type algebra.
\end{defn}

\begin{rem}
In the definition of locally compact operators, we could have required all $T_{x,y}$ to belong to $\overline{M}^p_S:=\overline{\bigcup_{n\in S}M_n(\C)}\subset B(\ell^p(S))$, which makes a difference only when $p=1$. This condition will not be satisfied by the $\ell^1$ uniform algebra $UB^1(X)$. Nevertheless, with this alternative definition, the results below (Theorem \ref{thm:Roetype} and Corollary \ref{cor:Roetype}) still hold.
\end{rem}

The $\ell^p$ uniform Roe algebra $B^p_u(X)$ is an example of an $\ell^p$ Roe-type algebra.
Recall the algebras $UB^p(X)$, $B^p(X)$, and $B^p_s(X)$ defined on page \pageref{egRoetype} after Definition \ref{pRoedef}, and that $B^p_s(X)\cong B^p_u(X)\otimes K(\ell^p(\N))$ for all $p\in[1,\infty)$. The algebras $UB^p(X)$ and $B^p(X)$ are also $\ell^p$ Roe-type algebras but $B^p_s(X)$ is not. All three algebras are coarse invariants while $B^p_u(X)$ is not. Indeed, if $X$ and $Y$ are coarsely equivalent, then the isometric isomorphism $\Phi:B(\ell^p(X,\ell^p))\rightarrow B(\ell^p(Y,\ell^p))$ in the proof of (1) $\Rightarrow$ (2) in Theorem \ref{thm2} restricts to isometric isomorphisms $B^p_s(X)\cong B^p_s(Y)$, $UB^p(X)\cong UB^p(Y)$, and $B^p(X)\cong B^p(Y)$.

On the other hand, note that all operators of the form $\xi\otimes\delta_{x_0,n_0}$, where $\xi\in\ell^p(X,\ell^p(S))$, $x_0\in X$, and $n_0\in S$, are contained in any $\ell^p$ Roe-type algebra, which enables one to show that Lemma \ref{spatially implemented 2} holds for $\ell^p$ Roe-type algebras by following the proof of Lemma \ref{spatially implemented}. One can then show that the implication (4) $\Rightarrow$ (1) in Theorem \ref{thm2} holds with $B^p_u(X)\otimes\overline{M}^p_\infty$ replaced by any $\ell^p$ Roe-type algebra or $B^p_s(X)$.

Thus, we obtain the following analogs of \cite[Theorem~4.1, Theorem~5.1 and Theorem~6.1]{MR3116573}, which deal with the $p=2$ case. It is worth noting that the metric spaces are assumed to have property A in \cite[Theorem~4.1 and Theorem~6.1]{MR3116573}.

\begin{thm} \label{thm:Roetype}
Let $X$ and $Y$ be metric spaces with bounded geometry, let $p\in[1,\infty)\setminus\{2\}$, and let $A^p(X)$ and $A^p(Y)$ be $\ell^p$ Roe-type algebras associated to $X$ and $Y$ respectively. If $A^p(X)$ and $A^p(Y)$ are isometrically isomorphic, then $X$ and $Y$ are coarsely equivalent. The same conclusion holds if $B^p_s(X)$ and $B^p_s(Y)$ are isometrically isomorphic.

The converse holds for the algebras $UB^p(X)$, $B^p(X)$, and $B^p_s(X)$. In particular, since $B^p_s(X)\cong B^p_u(X)\otimes K(\ell^p(\N))$ for $p\in[1,\infty)$, Theorem \ref{thm2} and Corollary \ref{cor:grps} also hold with $\overline{M}^p_\infty$ replaced by $K(\ell^p(\N))$.
\end{thm}

\begin{cor} \label{cor:Roetype}
Let $X$ and $Y$ be metric spaces with bounded geometry, let $p\in[1,\infty)$, and let $A^p[X]$ and $A^p[Y]$ be algebraic $\ell^p$ Roe-type algebras associated to $X$ and $Y$ respectively. If $A^p[X]$ and $A^p[Y]$ are isometrically isomorphic, then $X$ and $Y$ are coarsely equivalent. The same conclusion holds if $\mathbb{C}^p_s[X]$ and $\mathbb{C}^p_s[Y]$ are isometrically isomorphic.

The converse holds for the algebras $U\mathbb{C}^p[X]$, $\mathbb{C}^p[X]$, and $\mathbb{C}^p_s[X]$.
\end{cor}

Finally, let us return to the case $p=2$. We will denote the $\ell^2$ uniform Roe algebra of a metric space $X$ by $C_u^*(X)$ as in the literature.

Let $(X,d)$ be a (discrete) metric space with bounded geometry. For every $R>0$, we consider the $R$-neighbourhood of the diagonal in $X\times X$: $\Delta_R:=\{(x,y)\in X\times X: d(x,y)\leq R\}$. Define
\begin{align*}
G(X):=\bigcup_{R>0}\overline{\Delta}_R\subseteq \beta(X\times X).
\end{align*}
It turns out that the domain, range, inversion and multiplication maps on the pair groupoid $X\times X$ have unique continuous extensions to $G(X)$. With respect to these extensions, $G(X)$ becomes a principal, étale, locally compact $\sigma$-compact Hausdorff topological groupoid with the unit space $\beta X$ (see \cite[Proposition~3.2]{MR1905840} or \cite[Theorem~10.20]{Roe03}). Since $G(X)^{(0)}=\beta X$ is totally disconnected, $G(X)$ is also ample (see \cite[Proposition 4.1]{E10}). Moreover, the uniform Roe algebra $C_u^*(X)$ of $X$ is naturally isomorphic to the reduced groupoid $C^*$-algebra of $G(X)$, which maps $\ell^\infty(X)$ onto $C(G(X)^{(0)})$ (see \cite[Proposition~10.29]{Roe03} for a proof). We could compare the following proposition with \cite[Corollary~1.16]{WW}.

\begin{prop}\label{p=2 case}
Let $X$ and $Y$ be metric spaces with bounded geometry. Then the following are equivalent:
	\begin{itemize}
\item[(1)] $X$ and $Y$ are coarsely equivalent.
		\item[(2)] $G(X)$ and $G(Y)$ are equivalent as topological groupoids.
		\item[(3)] $G(X)\times \mathcal{R}$ and $G(Y)\times \mathcal{R}$ are isomorphic, where $\mathcal{R}$ denotes the full countable equivalence relation on $\N$. 
	\item[(4)] $(l^{\infty}(X)\otimes C_0(\N),C_u^*(X)\otimes K(\ell^2(\N))\cong (l^{\infty}(Y)\otimes C_0(\N),C_u^*(Y)\otimes K(\ell^2(\N)))$.
	
			\end{itemize}
\end{prop}
\begin{proof}
$(1) \Rightarrow (2)$: It follows from \cite[Corollary~3.6]{MR1905840}.

$(2) \Leftrightarrow (3)$: It follows from \cite[Theorem~2.1]{MR3601549}.

$(3) \Rightarrow (4)$: This is clear.

$(4) \Rightarrow (3)$: It follows from \cite[Proposition~4.13]{MR2460017}, which is also true without any change for locally compact $\sigma$-compact Hausdorff topologically principal \'etale groupoids (see also \cite[Remark~5.1]{XL18}).

$(3) \Rightarrow (1)$: In particular, $C_c(G(X)\times \mathcal{R})$ and $C_c(G(Y)\times \mathcal{R})$ are $*$-isomorphic. Since $C_c(G(X)\times \mathcal{R})$ and $C_c(G(Y)\times \mathcal{R})$ are respectively nothing but $\C_s[X]$ and $\C_s[Y]$ as in \cite[Example~2.2]{MR3116573}, we complete the proof by \cite[Theorem~6.1]{MR3116573}.
\end{proof}

{\bf Acknowledgments}. The second-named author would like to thank Christian B\"onicke, Hannes Thiel and Jiawen Zhang for helpful and enlightening discussions.

\bibliographystyle{plain}
\bibliography{bib}

\begin{thebibliography}{10}

\bibitem{AlbKal}
Fernando Albiac and Nigel Kalton.
\newblock {\em Topics in Banach Space Theory}.
\newblock Graduate Texts in Mathematics. Springer, 2nd edition, 2016.

\bibitem{ALLW17}
Pere Ara, Kang Li, Fernando Lled\'o, and Jianchao Wu.
\newblock Amenability and uniform {R}oe algebras.
\newblock {\em J. Math. Anal. Appl.}, 459(2):686--716, 2018.

\bibitem{BF18}
Bruno de~Mendon\c{c}a Braga and Ilijas Farah.
\newblock On the rigidity of uniform {R}oe algebras over uniformly locally
  finite coarse spaces.
\newblock {\em preprint}, 2018.
\newblock arXiv:1805.04236.

\bibitem{BNW}
Jacek Brodzki, Graham~A. Niblo, and Nick Wright.
\newblock Property {A}, partial translation structures, and uniform embeddings
  in groups.
\newblock {\em J. Lond. Math. Soc. (2)}, 76(2):479--497, 2007.

\bibitem{BGR}
Lawrence~G. Brown, Philip Green, and Marc~A. Rieffel.
\newblock Stable isomorphism and strong {M}orita equivalence of
  {$C\sp*$}-algebras.
\newblock {\em Pacific J. Math.}, 71(2):349--363, 1977.

\bibitem{BO}
Nathaniel Brown and Narutaka Ozawa.
\newblock {\em $C^\ast$-Algebras and Finite-Dimensional Approximations}.
\newblock Graduate Studies in Mathematics. American Mathematical Society, 2008.

\bibitem{MR3601549}
Toke~Meier Carlsen, Efren Ruiz, and Aidan Sims.
\newblock Equivalence and stable isomorphism of groupoids, and
  diagonal-preserving stable isomorphisms of graph {$C^*$}-algebras and
  {L}eavitt path algebras.
\newblock {\em Proc. Amer. Math. Soc.}, 145(4):1581--1592, 2017.

\bibitem{Cher}
Paul~R. Chernoff.
\newblock Representations, automorphisms, and derivations of some operator
  algebras.
\newblock {\em J. Funct. Anal.}, 12:275--289, 1973.

\bibitem{DF}
Andreas Defant and Klaus Floret.
\newblock {\em Tensor Norms and Operator Ideals}.
\newblock Number 176 in North-Holland Mathematics Studies. North-Holland
  Publishing Co., 1993.

\bibitem{MR2730576}
Tullia Dymarz.
\newblock Bilipschitz equivalence is not equivalent to quasi-isometric
  equivalence for finitely generated groups.
\newblock {\em Duke Math. J.}, 154(3):509--526, 2010.

\bibitem{EwertMeyer}
Eske~Ellen Ewert and Ralf Meyer.
\newblock Coarse geometry and topological phases.
\newblock {\em preprint}, 2018.
\newblock arXiv:1802.05579.

\bibitem{E10}
Ruy Exel.
\newblock Reconstructing a totally disconnected groupoid from its ample
  semigroup.
\newblock {\em Proc. Amer. Math. Soc.}, 138(8):2991--3001, 2010.

\bibitem{MR1876896}
Erik Guentner and Jerome Kaminker.
\newblock Exactness and the {N}ovikov conjecture.
\newblock {\em Topology}, 41(2):411--418, 2002.

\bibitem{GTY}
Erik Guentner, Romain Tessera, and Guoliang Yu.
\newblock A notion of geometric complexity and its application to topological
  rigidity.
\newblock {\em Invent. Math.}, 189(2):315--357, 2012.

\bibitem{MR1739727}
Nigel Higson and John Roe.
\newblock Amenable group actions and the {N}ovikov conjecture.
\newblock {\em J. Reine Angew. Math.}, 519:143--153, 2000.

\bibitem{MR0043392}
Richard~V. Kadison.
\newblock Isometries of operator algebras.
\newblock {\em Ann. Of Math. (2)}, 54:325--338, 1951.

\bibitem{MR3158244}
Julian Kellerhals, Nicolas Monod, and Mikael R{\o}rdam.
\newblock Non-supramenable groups acting on locally compact spaces.
\newblock {\em Doc. Math.}, 18:1597--1626, 2013.

\bibitem{Kub}
Yosuke Kubota.
\newblock Controlled topological phases and bulk-edge correspondence.
\newblock {\em Commun. Math. Phys.}, 349(2):493--525, 2017.

\bibitem{Lamp}
John Lamperti.
\newblock On the isometries of certain function-spaces.
\newblock {\em Pacific J. Math.}, 8:459--466, 1958.

\bibitem{LL}
Kang Li and Hung-Chang Liao.
\newblock Classification of uniform {R}oe algebras of locally finite groups.
\newblock {\em J. Operator Theory}, 80(1):25--46, 2018.

\bibitem{LW18}
Kang Li and Rufus Willett.
\newblock Low-dimensional properties of uniform {R}oe algebras.
\newblock {\em J. London Math. Soc.}, (2) 97:98–124, 2018.

\bibitem{XL18}
Xin Li.
\newblock Constructing {C}artan subalgebras in classifiable stably finite
  {$C^*$}-algebras.
\newblock {\em preprint}, 2018.
\newblock arXiv:1802.01190.

\bibitem{MR1763912}
Narutaka Ozawa.
\newblock Amenable actions and exactness for discrete groups.
\newblock {\em C. R. Acad. Sci. Paris S\'er. I Math.}, 330(8):691--695, 2000.

\bibitem{Pal}
Theodore~W. Palmer.
\newblock {\em Banach Algebras and the General Theory of *-Algebras}, volume~1.
\newblock Cambridge University Press, 1994.

\bibitem{Phil12}
N.~Christopher Phillips.
\newblock Analogs of {C}untz algebras on {$L^p$} spaces.
\newblock {\em preprint}, 2012.
\newblock arXiv:1309.4196.

\bibitem{Phil13}
N.~Christopher Phillips.
\newblock Crossed products of {$L^p$} operator algebras and the {$K$}-theory of
  {C}untz algebras on {$L^p$} spaces.
\newblock {\em preprint}, 2013.
\newblock arXiv:1309.6406.

\bibitem{Pitt}
H.R. Pitt.
\newblock A note on bilinear forms.
\newblock {\em J. Lond. Math. Soc.}, 11:174--180, 1936.

\bibitem{MR2460017}
Jean Renault.
\newblock Cartan subalgebras in {$C^*$}-algebras.
\newblock {\em Irish Math. Soc. Bull.}, 61:29--63, 2008.

\bibitem{Roe03}
John Roe.
\newblock {\em Lectures on coarse geometry}, volume~31 of {\em University
  Lecture Series}.
\newblock American Mathematical Society, Providence, RI, 2003.

\bibitem{MR2873171}
Mikael R{\o}rdam and Adam Sierakowski.
\newblock Purely infinite {$C^*$}-algebras arising from crossed products.
\newblock {\em Ergodic Theory Dynam. Systems}, 32(1):273--293, 2012.

\bibitem{Scarparo:2016kl}
Eduardo Scarparo.
\newblock Characterizations of locally finite actions of groups on sets.
\newblock {\em Glasg. Math. J.}, 60(2):285--288, 2018.

\bibitem{MR1905840}
Georges Skandalis, Jean-Louis Tu, and Guoliang Yu.
\newblock The coarse {B}aum-{C}onnes conjecture and groupoids.
\newblock {\em Topology}, 41(4):807--834, 2002.

\bibitem{MR2523336}
J{\'a}n {\v{S}}pakula.
\newblock Uniform {$K$}-homology theory.
\newblock {\em J. Funct. Anal.}, 257(1):88--121, 2009.

\bibitem{MR3116573}
J{\'a}n {\v{S}}pakula and Rufus Willett.
\newblock On rigidity of {R}oe algebras.
\newblock {\em Adv. Math.}, 249:289--310, 2013.

\bibitem{MR3557774}
J{\'a}n {\v{S}}pakula and Rufus Willett.
\newblock A metric approach to limit operators.
\newblock {\em Trans. Amer. Math. Soc.}, 369(1):263--308, 2017.

\bibitem{MR2800923}
ShuYun Wei.
\newblock On the quasidiagonality of {R}oe algebras.
\newblock {\em Sci. China Math.}, 54(5):1011--1018, 2011.

\bibitem{WW}
Stuart White and Rufus Willett.
\newblock Cartan subalgebras in uniform {R}oe algebras.
\newblock {\em preprint}, 2017.

\bibitem{Whyte}
Kevin Whyte.
\newblock Amenability, bilipschitz equivalence, and the von {N}eumann
  conjecture.
\newblock {\em Duke Math. J.}, 99(1):93--112, 1999.

\bibitem{MR2562146}
Rufus Willett.
\newblock Some notes on property {A}.
\newblock In {\em Limits of graphs in group theory and computer science}, pages
  191--281. EPFL Press, Lausanne, 2009.

\bibitem{WZ10}
Wilhelm Winter and Joachim Zacharias.
\newblock The nuclear dimension of {$C^\ast$}-algebras.
\newblock {\em Adv. Math.}, 224(2):461--498, 2010.

\bibitem{Yu95}
Guoliang Yu.
\newblock Coarse {B}aum-{C}onnes conjecture.
\newblock {\em K-Theory}, 9:199--221, 1995.

\bibitem{MR1451759}
Guoliang Yu.
\newblock Localization algebras and the coarse {B}aum-{C}onnes conjecture.
\newblock {\em $K$-Theory}, 11(4):307--318, 1997.

\bibitem{Yu00}
Guoliang Yu.
\newblock The coarse {B}aum-{C}onnes conjecture for spaces which admit a
  uniform embedding into {H}ilbert space.
\newblock {\em Invent. Math.}, 139(1):201--240, 2000.

\end{thebibliography}

\end{document}